\documentclass[12pt,reqno]{amsart}

\usepackage{epsfig}
\usepackage{amsmath, amsthm, amsfonts}
\usepackage{graphicx}
\usepackage{pgf}
\usepackage{hyperref}

\numberwithin{equation}{section}

\newcommand{\SLZ} {\rm {SL}_2(\mathbb{Z})}
\newcommand{\PSLZ}{\rm{PSL}_2(\mathbb{Z})}

\theoremstyle{plain}
\newtheorem{thm}{\protect\theoremname}
\providecommand{\theoremname}{Theorem}

\begin{document}

\author{Hartmut Monien}

\address{
  Bethe Center\\
  University Bonn\\
  Nussallee 12\\
  53115 Bonn\\
  Germany
}

\email{hmonien@uni-bonn.de}


\date{\today}

\title{The sporadic group $J_2$\\Hauptmodul and Bely{\u\i} map}

\begin{abstract}
  Determining Fourier coefficients of modular forms of a finite index
  noncongruence subgroups of $\PSLZ$ is still a non-trivial task
  \cite{LiLongYang}. In this brief note we describe a new algorithm to
  reliably calculate an approximation for a modular form of a given
  weight. As an example we calculate the hauptmodul and Bely{\u\i} map of a
  genus zero subgroup of the modular group defined via a canonical
  homomorphism by the second Janko group. Our main result is the field
  of definition of its Bely{\u\i} map and the Fourier coefficients of its
  hauptmodul.
\end{abstract}

\maketitle 

\section{Introduction}

In their paper on noncongruence subgroups Atkin and Swinnerton-Dyer
\cite{Atkin_Swinnerton-Dyer} gave a description of a practical
computational algorithm to obtain numerical approximations to modular
forms. Their ideas were further developed by Hejhal \cite{Hejhal},
Str\"omberg \cite{Stromberg}, Booker, Str\"ombergson and Venkatesh
\cite{Booker} to calculate Maass waveforms and by Selander and
Str\"ombergson \cite{SelanderStrombergson} to calculate a Bely{\u\i}
map \cite{BelyiA,BelyiB}.

In this brief note we describe an iterative algorithm based on these
ideas to calculate values of modular functions of any finite index
subgroup $\Gamma\subset\PSLZ$ of the modular group numerically to any
given precision and use it to obtain Bely{\u\i} maps. To demonstrate the
effectiveness of the algorithm we proof the following theorem
\begin{thm}
  Let $\sigma_0,\sigma_1 \in S_{100}$ be the permutation given in the
  appendix. Then the Janko group $J_2$ is generated by $\sigma_0$ and
  $\sigma_1$ and the triple
  $(\sigma_0, \sigma_1, (\sigma_0\sigma_1)^{-1})$ define up to
  simultaneous conjugation a subgroup $\Gamma$ of the full modular
  group $\PSLZ$ with a rational Bely{\u\i} map
  $\Phi:X(\Gamma)\rightarrow\mathbb{P}^1$ which obeys the equation
  $\Phi(z) = p_3(z)/p_c(z) = 1728 + p_2(z)/p_c(z)$ relating the
  branching at the elliptic points of order two and three with the
  polynomial $p_2$, $p_3$ and $p_c$ given in the accompanying
  material. These polynomials are defined over the number field
  $K=\mathbb{Q}[a]/(a^{10} - a^{9} - 2 a^{8} + 8 a^{7} - 14 a^{6} + 35
  a^{4} - 68 a^{3} + 89 a^{2} - 74 a + 23)$ with the Galois group
  $Gal(K/\mathbb{Q})=S_2 \wr S_5$ of order 3840 and discriminant
  $3^6 5^3 7^8$.
\end{thm}

\begin{proof}
  It is easy to check that the permutations $\sigma_0$ and $\sigma_1$
  fulfil the defining relations of the Janko group $J_2$,
  $\sigma_0^2=\sigma_1^3= \left(\sigma_0\sigma_1\right)^7=
  \left(\sigma_0\sigma_1\sigma_0\sigma_1^2\right)^{12}=1$.  The
  theorem $A$ of Magaard \cite{Magaard} states which of the sporadic
  groups with triple $(\sigma_0, \sigma_1, (\sigma_0\sigma_1)^{-1})$
  defines a subgroup $\Gamma$ of $\PSLZ$ with a Riemann surface
  $X(\Gamma)$ of genus zero.  In addition this is an admissible for
  the theorem of Millington \cite{Millington} and therefore define (up
  to simultaneous conjugation) an index 100 subgroup
  $\Gamma\subset\PSLZ$. Using theorem (3.1) of Hsu \cite{Hsu} it is
  easy to verify that the group is a noncongruence subgroup of
  $\PSLZ$. This group has no elliptic point of order 2 and 4 elliptic
  points of order 3 and 2 cusps of width 1 and 14 cusps of order 7.
  This defines the branching structure of the Bely{\u\i} map
  $\Phi:X(\Gamma)\rightarrow\mathbb{P}^1$ and the factorization of the
  polynomials. Using the explicit form of the polynomials an easy
  calculation shows that $p_3(z)/p_c(z) = 1728 + p_2(z)/p_c(z)$ which
  completes the proof.
\end{proof}

A number of major improvements make the algorithm numerically stable
and efficient enough to use arbitrary precision arithmetic. The new
ideas are: using domain decomposition techniques for the modular
tessellation, fast Fourier transform to precondition the resulting
linear system, Krylov subspace methods to solve the resulting linear
equations iteratively.  The algorithm can easily be extended to
calculate cusp forms and vector valued modular forms. We have
implemented our algorithm in {\em Haskell} \cite{Haskell} which allows
to easily write generic arbitrary precision code and to parallelize
the domain decomposition.

The determination of large zero Bely{\u\i} maps is in general a
difficult problem \cite{Zvonkin}. The subject has recently attracted
interest and some substantial results were obtained for genus zero
Belyi maps of degree up to 250 \cite{Wuerzburg}, the HS sporadic
group\cite{WuerzburgHS}, for Hurwitz groups \cite{DavidRoberts} and
composite genus 1 Belyi maps \cite{Vidunas}.  For a detailed review of
previously known methods and some new techniques and results see
\cite{SisjlingVoight} and my lecture at the Euler Institute 2014
\cite{StPetersburgTalk} for a more detailed account of complex
analytic methods.

\section{Basic Setup}
Let $\PSLZ = \SLZ/\{\pm 1\}$ be the projective special linear group over
$\mathbb{Z}$ acting on the complex upper half plane
$\mathcal{H} = \{x+iy\in\mathbb{C}\;|\;y >0\}$ by
M\"obius transformations $\gamma\in\PSLZ$
\[
  \gamma
  =\left(\begin{array}{cc}a&b\\c&d\end{array}\right):
  \mathcal{H}\rightarrow \mathcal{H},\  
  z\mapsto\gamma z = \frac{a z + b}{c z +d}.
\]
Let $G$ be a finitely presented group and $H$ a finite index subgroup
of $G$. The generators of $G$ induce a permutation of the right coset
$H\backslash G$. Then the group $H$ is determined up to isomorphism by
the permutations induced by the action of the generators of $G$ on the
right coset $H\backslash G$. This fact allows for the classification
of all finite index subgroups of $\PSLZ$.  A theorem by Millington
\cite{Millington} states that up to simultaneous conjugation there is
a one to one correspondence between triples of permutations
$\{(\sigma_0, \sigma_1,\sigma_\infty)\in
S^3_d\;|\;\sigma_0^2=\sigma_1^3=\sigma_0\sigma_1\sigma_\infty=id\}$
with the finite index subgroups of $\PSLZ$ where $S_d$ is the
symmetric group of $d$ objects (in this case cosets).  Let
$\Gamma\subset\PSLZ$ be a finite index subgroup of the modular group
of index $d$ defined by a triple of permutations as above.  If
$\Gamma$ contains
$ \left\{\gamma \in \PSLZ\;|\:\gamma = \pm I_2\mod N \right\} $ for
some fixed integer, $N$, it is called a congruence subgroup otherwise it is
called a noncongruence subgroup. Both cases can be distinguished
easily by testing relations of the generators \cite{Hsu}.  A
fundamental domain $\mathcal{F}(\Gamma)$ of $\Gamma$ is a subset of
$\mathcal{H}$ such that for all $z\in\mathcal{H}$ there is exactly one
element $\gamma\in\Gamma$ for which $\gamma z
\in\mathcal{F}(\Gamma)$. Here we choose
$\mathcal{F}=\left\{z=x+iy\in\mathbb{C}\;|\;x\in[-1/2, 1/2), |z| \ge
  1\right\}$ as fundamental domain for the full modular group. Since
$\Gamma$ is a subgroup of finite index $d$ we have
$\PSLZ = \cup^{d}_{i=1} \Gamma \gamma_i$ where
$\{\gamma_1,\gamma_2\ldots \gamma_d\}\subset\PSLZ$ is a list of right
coset representatives which implies that
$\mathcal{F}(\Gamma)=\cup_{i=1}^d\gamma_i\mathcal{F}$ is a fundamental
domain of $\Gamma$.

In general $\Gamma$ will contain a number of finite index subgroups of
of the Borel subgroup with fix points in $\mathbb{P}^1(\mathbb{Q})$
called cusps. Two cusps related by a $\gamma\in\Gamma$ are called
equivalent. Each cycle of $\sigma_\infty$ gives rise to a cusp
(equivalence class) with the corresponding cusp width being defined as
the length of the cycle. Let $k\in\{0,..\kappa\}$ enumerate the cycles
and let $\{x_0,x_1..x_\kappa\}$ be a list of unique inequivalent cusp
representatives and $w_k$ be the width of cusp $x_k$.  Let
$\mathcal{H^*} = \mathcal{H}\cup\mathbb{P}^1(\mathbb{Q})$ be the
extended complex upper half plane. A modular function is a complex
valued function that extends to a meromorphic function on the
compactified upper half plane $f:\mathcal{H}^*\rightarrow\mathbb{C}$
satisfying $f(\gamma z) =f(z)$ for every $\gamma$ and every
$z\in\mathcal{H}$.  The genus of $\Gamma$ is defined to be the genus
of the Riemann surface $X(\Gamma)=\Gamma\backslash\mathcal{H}^*$.
Here we focus on the genus zero case where the field of modular
functions is generated by one modular function usually called
hauptmodul. We choose one of the cusps with the smallest cusp width,
$w_0$, as principal cusp and fix the generator
$j_\Gamma:\mathcal{H}^*\rightarrow\mathbb{C}$ uniquely by imposing a
growth condition on it so that as $y\rightarrow\infty$ it grows like
$j_\Gamma(z=x+i y) = q^{-1} + 0 + O\left(q\right)$ where
$q = \exp(2\pi i z/w_0)$ and stays finite at all the other cusps.  The
rational map $\Phi:X(\Gamma)\rightarrow\mathbb{P}^1$ is of degree $d$
and a famous theorem by Bely{\u\i} \cite{BelyiA,BelyiB} asserts that
it can be defined up to isomorphism over $\overline{\mathbb{Q}}$.
Once $j_\Gamma(z)$ is determined fix an explicit representation of
$\Phi$ relating the modular $j:\mathcal{H}\rightarrow\mathbb{C}$
invariant of the full modular group to $j_\Gamma$ so that we have
$\Phi\left(j_\Gamma(z)\right) = j(z)$ for all $z\in\mathcal{H}$. We
choose the normalization of the modular $j$ invariant such that at the
elliptic point of order two the modular invariant takes on the value
$j(i)=1728$ and vanishes at the elliptic point of order three
$j((1+i\sqrt{3})/2) = 0$.
 
Let $k$ enumerate the cusps with $k=0$ being the principal
cusp. Define $\nu_0 = T^{w_0}$ where
\[
  T = \left(\begin{array}{cc}1&1\\0&1\end{array}\right)
\]
and for the remaining cusps $x_k=(p,q)$, with $\gcd(p,q)=1$, pick a
fixed $p'$ from the solutions of the congruence $p p'= -1\mod q$ to
define \cite{Zuckerman}
\[
  \nu_k=\left(\begin{array}{cc}-p'&-(p p' + 1)/q\\q&-p\end{array}\right).
\]
It is easy to see that $\nu_k$ sends $x_k$ to infinity and
$\nu_k^{-1} T \nu_k$ stabilizes $x_k$.  The modular function can therefore
be expanded at each cusp in a Poincar\'e series
\cite{Zuckerman,SelanderStrombergson}
\begin{equation}
  \label{eq:1}
  j_\Gamma(z) = \sum_{m=\mu_k}^\infty a^{(k)}_m q_k^m   
\end{equation}
with $\mu_0=-1$ and $\mu_k=0$ for $k>0$,
$q_k=\exp\left(2\pi i z'/w_k\right)$, $z' = \nu_k z$ and (unknown)
Fourier coefficients $a_m^{(k)}$. Here we set the principal part
$a^{(0)}_{-1}=1,$ and fix the arbitrary constant term $a^{(0)}_0=0$ in
agreement with the growth condition. The Fourier coefficients are
known to have at most polynomial growth with $m$. The Poincar\'e
series for different cusps are obviously not independent.  Modular
transformations $\gamma\notin\Gamma$ act transitively on the cusps. 
Zuckerman \cite{Zuckerman} used this fact to generate linear relations
between the expansions on different cusps for general finite index subgroup
$\Gamma\subset\PSLZ$.

\section{Description of the numerical algorithm}

From an numerical analysis point of view determining a numerical
approximation for a modular function is equivalent to solving two
Laplace equations for two real functions on an infinitely extended
domain with complicated ``periodic'' boundary conditions (coming from
the generators of the finite index subgroup).  Domain decomposition is
a technique for the large scale numerical solution of boundary value
problems of elliptic partial differential equations. It goes back to
an 1870 paper by H. Schwarz \cite{Schwarz} and is one of several
constructive techniques of conformal mapping he developed for the
uniformization problem. The basic idea is to solve a Dirichlet problem
for an elliptic partial differential equation on a complicated domain
by decomposing it into simpler overlapping domains on which the
Dirichlet problem is easily solved. Its convergence for general second
order differential equations has been proven by Mihlin\cite{Mihlin}
for very general domains. We now explain how to apply these ideas to
the calculation of modular functions.
 
For each cusp with label $k$ we define a domain
$\mathcal{F}_k$ of the complex upper half plane by
\[
  \mathcal{F}_k = 
  \left\{x+iy\in\mathbb{C}\;|\;  x\in[-1/2,w_k-1/2),
    y\in[1/2,\infty)\right\}.
\]
At this point the important observation here is that the preimage of
$\cup_k\nu_k^{-1}\mathcal{F}_k$ form an overlapping domain
decomposition of the fundamental domain $\mathcal{F}(\Gamma)$.  The
domain $\mathcal{F}_k$ contains the cosets of the cycle (cusp) $k$ as
well as parts of its neighboring cosets determined by $\sigma_0$
and $\sigma_1$.  We solve the Laplace equation in the domain
$\mathcal{F}_k$. The periodic boundary conditions and the boundary
condition as $y\rightarrow\infty$ are taken care of by the particular form of
the Fourier expansion. The Dirichlet boundary condition is given by
the values of the hauptmodul calculated at the lower boundary of the
domain which are determined by the expansions in the domains
$\mathcal{F}_{k'}$ with $k'\ne k$ and the upper boundary. We
approximate the modular function by a truncated expansions
\begin{equation}
  \label{eq:1}
  j^{(k)}_\Gamma(z) = \sum_{m=\mu_k}^{N_k} a^{(k)}_m q_k^m   
\end{equation}
where we choose $N_k = N w_k$ with $N$ fixed. Suppose now that we
already have obtained some approximation for the Fourier coefficients
at all cusps. For any given $z$ in $\mathcal{H}$ we can determine
algorithmically a unique coset representative $\gamma_j$ such that
$\gamma_l z\in\mathcal{F}$. The label $l$ can be found uniquely in one
of the cycles of $\sigma_\infty $, say in cycle $k$, since the
permutation group generated by $\sigma_0$ and $\sigma_1$ acts
transitively. Then we approximate $j_\Gamma(z)$ by
$j_\Gamma^{(k)}(z)$. For each $k$ we evaluate the approximate modular
function according to the procedure outlined above at points
$z^{k}_j = \nu_k^{-1} (w_k j/M + i/2-1/2)$ for $j\in\{0,1..M-1\}$
where $M$ is chosen such that $M$ is a power of two with integer
exponent and $M>2\times N w_k$. The first condition ensures that we
can use the simplest form of fast Fourier transform to obtain the
Fourier coefficients while the second condition avoids aliasing
effects. We then apply FFT to determine the expansion coefficients of
$j_\Gamma^{(k)}$ for every domain.  This completes one step of the
domain decomposition iteration.  We can view the procedure as
iterative solution of the linear equations for the Fourier
coefficients with the inhomogeneous part arising from the principal
cusp. The rate of convergence is determined by the largest eigenvalue
of the (linear) iteration operator which is less than one if the
overlap of the domains is finite \cite{Mihlin} which implies linear
convergence. We observe linear convergence when using the simple
(Picard) iteration. The advantage of the procedure described above is
that convergence is guaranteed and it can easily be implemented in
multiprecision arithmetic \cite{MPFR}. A further advantage is that it can be used
as a pre-conditioner for more sophisticated iterative techniques which
go under the name of Krylov methods \cite{Trefethen,Saad} to solve the
linear equation. Using GMRES \cite{Saad} which we implemented in
multiprecision we actually do observe quadratic convergence in all
cases we have investigated. One important point is that compared to
direct methods the operation count is substantially reduced.

\section{Appendix}
The permutations of Theorem (1) are defined by (data from the Atlas
finite groups \cite{Atlas}):
\begin{eqnarray*}
  \sigma_0 &=& (1,84)(2,20)(3,48)(4,56)(5,82)(6,67)(7,55)(8,41)(9,35)(10,40)\\
    &&(11,78)(12,100)(13,49)(14,37)(15,94)(16,76)(17,19)(18,44)(21,34)\\
    &&(22,85)(23,92)(24,57)(25,75)(26,28)(27,64)(29,90)(30,97)(31,38)\\
    &&(32,68)(33,69)(36,53)(39,61)(42,73)(43,91)(45,86)(46,81)(47,89)\\
    &&(50,93)(51,96)(52,72)(54,74)(58,99)(59,95)(60,63)(62,83)(65,70)\\
    &&(66,88)(71,87)(77,98)(79,80)\\
  \sigma_1 &=&(1,80,22)(2,9,11)(3,53,87)(4,23,78)(5,51,18)(6,37,24)(8,27,60)\\
    &&(10,62,47)(12,65,31)(13,64,19)(14,61,52)(15,98,25)(16,73,32)\\
    &&(17,39,33)(20,97,58)(21,96,67)(26,93,99)(28,57,35)(29,71,55)\\
    &&(30,69,45)(34,86,82)(38,59,94)(40,43,91)(42,68,44)(46,85,89)\\
    &&(48,76,90)(49,92,77)(50,66,88)(54,95,56)(63,74,72)(70,81,75)\\
    &&(79,100,83)
\end{eqnarray*}
The polynomials $p_2$, $p_3$ and $p_c$ can be found in the data files
``p2.txt'', ``p3.txt'' and ``pc.txt'' accompanying the paper.
\end{document}